\documentclass[12pt]{amsart}     

\usepackage{amssymb}
\usepackage{amsmath}
\usepackage{amsthm}

\usepackage{esint}

\usepackage[T2A]{fontenc}
\usepackage[cp1251]{inputenc}
\newtheorem {theorem} {Theorem}

\newtheorem {corollary} [theorem]{Corollary}
\newtheorem {lemma}  [theorem]{Lemma}
\theoremstyle{definition}
\newtheorem{df}{Definition}

\theoremstyle{remark}

\newcommand{\rd}{\mathbb{R}^d}

\begin{document}

\title{Carleson  families of  cubes related to porous sets}


%\title running{}        % if too long for running head

\author{Andrei V.~Vasin }

\address{Admiral Makarov State University of Maritime and Inland Shipping,
Dvinskaya st.~5/7, St.~Petersburg 198035, Russia}
%\address{St.~Petersburg Department
%of Steklov Mathematical Institute, Fontanka 27, St.~Petersburg
%191023, Russia}
\email{andrejvasin@gmail.com}

%\thanks{This research was supported by the Russian Science Foundation (grant No.~23-11-00171).}



\begin{abstract}

 Given  a  porous  set $E\in \mathbb{R}^d$ and a dyadic lattice $\mathcal{D}$, we refine   the Carleson packing condition  and the sparseness property for the  dyadic cover  $\mathcal{D}_E=\{Q \in \mathcal{D}: \:  Q \cap E \neq \varnothing\}$.
  We study the inverse problem,  when a Carleson  family $\mathcal{S} \subset \mathcal{D}$ generates the porous set $E$ such that  $\mathcal{S} \subset \mathcal{D}_E$.

\end{abstract}

\maketitle

\section{Introduction}\label{s_int}
\subsection{Background}
 Studying  the  weak embedding property  of singular inner functions,  Borichev, Nicolau and Thomas  \cite[Lemma 7]{BNT}  found the Carleson packing condition of the family $\mathcal{D}_E$ of dyadic intervals intersecting a porous set $E$  in the unit circle $ \mathbb{T} $.
The  condition which holds in  $ \rd$ rather than in $ \mathbb{T} $ \cite[Theorem 7]{V}  is equivalent   to  the sparseness of this family ( Verbitsky  \cite{VI}).
     The latter  refers to the sparse domination principle for the Calder\'{o}n-Zygmund
 operators etc (see, for instance,
  Lerner-Nazarov\cite {LN} and reference therein).

  In the paper  we  strengthen   the sparseness  property for  $\mathcal{D}_E$,  proving that  one can choose a  disjoint collection   of  cubes as a  disjoint collection of measurable sets; see Theorem \ref{t2} below.
Furthermore, we   characterize the  Carleson packing condition of the family $\mathcal{D}_E$ with respect to the weighted measure
 $\mathrm{dist}^{-\alpha }(x, E)dx $. This one concerns a result obtained by Ivrii and Nicolau \cite[Lemma 3.1]{IN} in one dimension.

    Our motivation comes from   a result  of Ihnatsyeva and V\"{a}h\"{a}kangas \cite[Theorem 2.10]{IV}  for the families      of cubes
\[\mathcal{D}_{\gamma, E}=\{Q \in \mathcal{D}: \:\mathrm{dist} (Q, E)< \gamma \ell(Q)\} \]
with porous   $E$ and  $\gamma >0$. In  \cite{IV} the authors research the  traces of  Triebel--Lizorkin spaces on  Ahlfors--David $\theta$-regular sets $E \in \rd$, which are known to be porous  for   $\theta < d$  by \cite{Mat}.
We study the sparseness and the Carleson properties of  that family (see Theorem \ref{t3},  Corollary \ref{cor3}).

\subsubsection{}
The paper is organized as follows. In Section 1.2--1.3 we recall the notation of dyadic decomposition,    the Carleson packing condition and porous sets.  The main results are formulated in Section 1.4.   The proofs are given  in Section 2.

 \subsection{Dyadic decomposition and Carleson families}
 Throughout this paper, we consider $\rd$ equipped with  the
$d$-dimensional Lebesgue  measure $dx$.    The Lebesgue  measure of $E$ is denoted by $|E|$.
  From the point of view of our proof it is convenient to use  $l_{\infty}$-metric.
   That is, for sets $E,\; F \subset \rd$ we put
   \[\mathrm{dist} (F, E)= \inf_{ x\in F,
  y \in E} \quad \max_{1\leq i \leq d} |x_i-y_i|. \]

We assume that a  cube is half-open and has sides parallel to the
coordinate axes. That is, a cube in $\rd$ is a set of the form
\[Q = [a_1; b_1)\times \dots   \times [a_d; b_d)\]
with side length $\ell(Q) = b_1-a_1 = \dots = b_d-a_d$ and with Lebesgue  measure $|Q|=|b_1-a_1|\times \dots   \times |b_d-a_d|$.
The dyadic decomposition of a cube $R\subset \rd$ is
\[ \mathcal{D}(R)= \bigcup_{j\geq 0} \mathcal{D}_j(R)\]
where each $\mathcal{D}_j(R)$ consists of the $2^{jd}$  pairwise disjoint (half-open) cubes $Q$, with side length
$\ell(Q) = 2^{-j}\ell(R)$, such that
\[R= \bigcup_{Q \in \mathcal{D}_j(R)} Q\]
for every $j =0,1, \dots $. The cubes in  $\mathcal{D}(R)$ are called dyadic cubes or children (with respect to $R$).

Define a dyadic lattice (with respect to $\rd$) by one of definitions in  \cite[\S 2]{LN}.
\begin{df}
  A dyadic lattice $\mathcal{D}$ is any collection of cubes such that
  \begin{enumerate}
    \item[DL1.]  If $ Q \in \mathcal{D}$, then $\mathcal{D}(Q) \subset \mathcal{D}$.

    \item[DL2.] Every 2 cubes $Q', Q''\in  \mathcal{D}$ have a common ancestor, i.e., there
exists $Q\in  \mathcal{D}$ such that $Q', Q''\in \mathcal{D}(Q)$.
    \item[DL3.] For every compact set $K \subset \rd$, there exists a cube $Q\in  \mathcal{D}$
containing $K$.
  \end{enumerate}
\end{df}
In what follows we will use the characterizing property of cubes $Q, \; R \in \mathcal{D}$:
\begin{equation}\label{dyp}
  Q\cap R= \{Q,\;R,\;  \varnothing\}.
\end{equation}

Let  $Q \in \mathcal{D}$. Then there exists a unique dyadic cube $\pi Q \in \mathcal{D}$
such that  $Q \in  \mathcal{D}_1(\pi Q)$, and $|\pi Q|=2^d|Q|$.

  \subsubsection{Carleson families. }A collection of cubes $\mathcal{S}\subset \mathcal{D}$ is called Carleson (satisfies the $\xi$-Carleson packing condition, $\xi >0$) if for each cube $R \in \mathcal{D} $
  \begin{equation}\label{e_carl}
    \sum _ {\substack {Q \in \mathcal{S}\\Q \subset R }}|Q|\leq \xi |R|
  \end{equation}
   \subsubsection{Sparse families. }A collection of cubes $\mathcal{S}\subset \mathcal{D}$ is called sparse  with the constant  $\lambda > 1$,
 if one can choose pairwise disjoint measurable sets
$S(Q) \subset Q \in \mathcal{S}$  such that  $\lambda |S(Q)|\geq  |Q|$ and $S(Q')\cap S(Q'')=\varnothing$ if $Q' \neq Q''$.
     It is known that a family $\mathcal{S}$   is $\xi$-Carleson if and only if it  is $\lambda$-sparse, $\lambda = \xi$ (see, for instance, \cite{VI},  \cite{ LN}) .

Let $E$ be a nonempty set. Cubes $Q' \in \mathcal{D}$ such that
 $ Q' \cap E = \varnothing$ will be called free cubes.
  For each cube $R\in \mathcal{D}$  let
\[\mathcal{D}_E(R)=\{Q \in \mathcal{D}(R): \:  Q \cap E \neq \varnothing\}\]
be a family of all dyadic cubes intersecting $E$,
  and let
\[\mathcal{F}_E(R)= \{Q' \in \mathcal{D}(R): Q' \cap E = \varnothing, \;Q'\neq R,\;  \pi Q'\cap E \neq \varnothing\}\]
 be a family of all maximal   pairwise disjoint dyadic free cubes.
Put
  \[\mathcal{D}_E= \bigcup _{R\in \mathcal{D}} \mathcal{D}_E(R). \]
Introduce a   version of the sparseness property  when  one can choose a disjoint family of cubes as a disjoint family of measurable sets.
  \begin{df}
  A collection of cubes $\mathcal{S}\subset \mathcal{D}$ is called well-sparse
 if there exist  a constant $\lambda >1$ and  a family of  pairwise disjoint cubes
$M(Q) \subset Q$, $Q \in \mathcal{S}$  such that  $\lambda \;|M(Q)|\geq  |Q|$.

  Let a set $E\subset \rd$.  If in the definition of well-sparse family, moreover, one asks     $M(Q)\cap E=\varnothing$   for each  $Q \in \mathcal{S}$, then $\mathcal{S}$ is called well-sparse  with respect to $E$.
   \end{df}

 \subsection{Porous sets }
\begin{df}
  Let $E$ be a nonempty set.
  A set $E\subset\mathbb{R}^d$ is called  porous ($\eta$-porous), if there  is a constant $\eta > 1$ such that for each dyadic cube $R\in\mathcal{D} $ there exists a free cube     $M(R) \subset R$ such that $\eta |M(R)|\geq  |R|$.
   \end{df}
  Instead of dyadic cubes, also general cubes  could be used in the definition of
porosity. However, the dyadic formulation  is convenient from the point of view of our
proofs.
The following properties  are easy to verify using the definition of porosity:
\begin{enumerate}
   \item [P1.] $E\subset \rd$ is   porous if and only if the closure  $\overline{E}$ is porous.
  \item [P2.]If $E\subset \rd$ is  porous, then $|E|=0$. This is a consequence of the Lebesgue
differentiation theorem.
  \item [P3.]Porous set  is nowhere dense.
\end{enumerate}
  \subsubsection{  Dyn'kin's inequality.}       The set   $E \subset \rd$, $|\overline{E}|=0$, is porous (see  \cite{D} in $\mathbb{T}$   or \cite{Dy} in $\rd$) if and only if
  there exist two constants  $0<\alpha <d$ and $C >0$ such that for each cube $R \in \mathcal{D}_E$ one has
 \begin{equation}\label{aaa}
   \sum_{Q' \in \mathcal{F}_E(R)}|Q'|^{1-\alpha /d}\leq C |R|^{1-\alpha /d}.
 \end{equation}
It holds an  integral version of (\ref{aaa}) with respect to the weighted measure
% $d\mu_{\alpha,E}=\mathrm{dist}^{-\alpha }(x, E)dx $ 
  \begin{equation}\label{e_i}
 \mu_{\alpha,E}(R)= \int_{R}\mathrm{dist}^{-\alpha }(x, E)dx \leq C |R|^{1-\alpha/d}.
 \end{equation}

   The supremum of those $\alpha \geq 0$ for which  inequalities (\ref{aaa}) or   (\ref{e_i})
hold with a constant $C=C(E, \alpha)>0$,  equals the   Aikawa--Assouad codimension $\mathrm{codim}_A E$ (see, for instance, \cite{Dy}, \cite{LT}).  Observe that   a set $E \subset \rd$ is porous if and only if $\mathrm{codim}_A E >0$.
\subsubsection{  Carleson property of porous $E$.} The set   $E \subset \rd$, $|\overline{E}|=0$, is porous
if and only if (see  \cite{BNT} in $\mathbb{T}$   or \cite{V} in $\rd$) the collection  $\mathcal{D}_E$ is Carleson.

 \subsection{Main results}

\subsubsection{} The improvement of the Carleson packing property of $\mathcal{D}_E$ for a porous set $E$ is the next theorem. We notice that (ii) is a strengthening of the sparseness property of $\mathcal{D}_E$ (see Section 1.2.2). Further, (iii)  means the uniform $\alpha$-Carleson-Beurling  condition (see \cite[Lemma 3.1]{IN}). Finally, (iv) is  the Carleson packing condition of $\mathcal{D}_E$ with respect to the weighted measure
 $\mu_{\alpha,E}$ in (\ref{e_i}).

\begin{theorem} \label{t1}
Let    $E\subset\mathbb{R}^d$ and  let  $\mathcal{D}$  be a dyadic lattice.  The following properties are equivalent.
\begin{enumerate}

   \item[(i)]
    $E$ is porous.
 \item[(ii)]
 The family $\mathcal{D}_E$  is  well-sparse with respect to $E$.
 \item[(iii)]  There exist constants $ \alpha > 0 $ and $ \zeta >0$ such that  for each cube $R \in \mathcal{D}$
    \begin{equation}\label{e_eqvi}
 \sum_{Q \in \mathcal{D}_E (R)}|Q|^{1-\alpha /d}\leq \zeta\; |R|^{1-\alpha /d}.
  \end{equation}
   \item[(iv)]    $|\overline{E}|=0$ and for each cube $R \in \mathcal{D}$
     \begin{equation}\label{e_dyn}
  \sum _ {Q \in \mathcal{D}_E(R)}\mu_{\alpha,E}(Q)\leq \xi\;\mu_{\alpha,E}(R)
  \end{equation}
  with   constants $\alpha > 0$ and $\xi>0$ independent of $R$.
   \end{enumerate}
\end{theorem}
Notice that  for some problems
(\ref{e_dyn}) is preferable to (\ref{e_eqvi}) because of the Carleson  condition on measure,
but   (\ref{e_eqvi}) has the advantage that it  contains terms $|Q|^{1-\alpha /d}$   independent  of  $E$,  whereas (\ref{e_dyn})
is not. See further Corollary \ref{cor5}.

The proof of Theorems \ref{t1}   yields   (\ref{e_eqvi}) and  (\ref{e_dyn})  for each
$ \alpha< \mathrm{codim}_A E$. We have an additional  characterization  of the Aikawa--Assouad codimension.
 \begin{corollary}\label{eqvi}
   Let  $E\subset \rd$ be a set with  $|\overline{E}|=0$. Then    $\mathrm{codim}_A E$  can be calculated as
    the supremum either  of those $\alpha \geq 0$ such that   property  (\ref{e_eqvi}) (or  (\ref{e_dyn}))
              holds.
       \end{corollary}

\subsubsection{}
Consider the inverse problem.
  Let a collection of cubes $\mathcal{S}\subset \mathcal{D}$ satisfy to the   Carleson packing condition. We want to find a set $E$ such that $ \mathcal{S} \subset \mathcal{D}_E$. Clearly, if such a set exists,  then for each $Q \in \mathcal{S} \subset \mathcal{D}_E$ it    necessarily implies
$\pi Q \in \mathcal{D}_E$.   It turned out  that if we assume  $\pi Q \in \mathcal{S}$ for each $Q \in \mathcal{S}$, then this  regularity type condition  is sufficient.
\begin{theorem} \label{t2}
 Given  a Carleson collection of cubes $\mathcal{S}\subset \mathcal{D}$ such that  $\pi Q \in \mathcal{S}$ for each cube $Q\in \mathcal{S}$. Then there is a porous set $E$ with
    $ \mathcal{S} \subset \mathcal{D}_E$.
\end{theorem}
   Applying the proof of Theorem \ref{t1}, we establish the related  result omitting an auxiliary    set $E$.

    \begin{corollary}\label{cor5}
       Given a Carleson collection of cubes $\mathcal{S}\subset \mathcal{D}$ such that    $\pi Q \in \mathcal{S}$ for each   cube $Q\in \mathcal{S}$. Then $\mathcal{S}$  is well-sparse and there exist two constants $ \alpha >0$ and $\zeta >0$ such that  for each cube $R \in \mathcal{D}$
      \[ \sum_{\substack {Q \in \mathcal{S}\\ Q\subset R}}|Q|^{1-\alpha /d}\leq \zeta \;|R|^{1-\alpha /d}.\]
        \end{corollary}
        \subsubsection{}
Applications to the larger families  $\mathcal{D}_{\gamma, E}$ introduced in  \cite{IV}.
  \begin{df}
     For  $E \subset \rd$   and a positive constant $\gamma$ define  the family of cubes, which are relatively close  to the set $E$:
     \begin{equation}\label{eee}
       \mathcal{D}_{\gamma, E}=\{Q \in \mathcal{D}: \: \mathrm{dist}(Q, E) < \gamma \ell(Q)\}.
     \end{equation}
  \end{df}
    Clearly,
 $ \mathcal{D}_E  \subset \mathcal{D}_{\gamma, E}$.
 We obtain  the characterization of a porous set $E$ in terms of the family  $\mathcal{D}_{\gamma, E} $.

\begin{theorem} \label{t3}
Let   $E\subset\mathbb{R}^d$ be a set and $\gamma >0$.  The following properties are equivalent.
\begin{enumerate}
  \item[(i)] The set $E$ is porous.
  \item[(ii)] The family
 $\mathcal{D}_{\gamma, E}$  satisfies the   Carleson packing condition.
 \item[(iii)] The family
 $\mathcal{D}_{\gamma, E}$  is well-sparse with respect to $E$.
      \end{enumerate}
   \end{theorem}
Theorem \ref{t1} and Theorem \ref{t3} imply additional properties of a porous set.
\begin{corollary}\label{cor3}
Let   $E\subset\mathbb{R}^d$ be a set and $\gamma >0$. The porosity of    $E\subset\mathbb{R}^d$ is  equivalent to any of following assertions.
  \begin{enumerate}
   \item[(i)] There exist constants $\alpha >0$ and $\zeta>0$ such that  for each cube $R \in \mathcal{D}$
 \[ \sum_{Q \in \mathcal{D}_{\gamma, E} (R)}|Q|^{1-\alpha /d}\leq \zeta |R|^{1-\alpha /d}.\]
   \item[(ii)]    $|\overline{E}|=0$ and the Carleson packing condition of $\mathcal{D}_{\gamma, E}$ with respect to the weighted measure
 $d\mu_{\alpha,E}=\mathrm{dist}^{-\alpha }(x, E)dx $ holds. Namely,
    there exist the  constants $\alpha > 0$ and $\xi>0$ such that for each cube $R \in \mathcal{D}$

\[ \sum _ {Q \in \mathcal{D}_{\gamma, E}(R)}\mu_{\alpha,E}(Q)\leq \xi\;\mu_{\alpha,E}(R).\]
\item[(iii)] Dyadic Carleson embedding inequality. There exists $\alpha >0$ such that for each  $1\leq p< \infty$  and arbitrary sequence $\{a_Q: \; Q \in \mathcal{D}_{\gamma, E}(R) \}$ of non-negative scalars
\begin{equation}\label{e_t2}
  \bigg \|\sum_{Q \in \mathcal{D}_{ \gamma, E}(R)} a_Q \chi_Q\bigg\|_{p,\alpha} \leq C
\bigg
\|\sup_{Q \in\mathcal{D}_{\gamma, E}(R)} a_Q \chi_Q \bigg\|_{p,\alpha},
\end{equation}
where $C=C(\alpha, p, \gamma, E)$, $\chi_Q$ is the characteristic function of a cube $Q$ and  $L^p$-norm $\|f\|_{p,\alpha}$ is defined with respect to the weight measure
     $ d\mu_{\alpha,E}(x)$.
 \end{enumerate}
\end{corollary}

Observe that   Corollary \ref{cor3} holds for each
$0 < \alpha< \mathrm{codim}_A E$. Item (iii) refines  \cite[Theorem 2.10]{IV} up  to the Carleson embedding inequality (see \cite[p. 59, Lemma 5]{MC}).

 .

\section{Proofs}
\subsection{ Proof   of Theorem \ref{t1}}
 \subsubsection{Proof of (i)$\Leftrightarrow$ (ii). }

  By the definition of the porosity, it would be easy to construct the family of free cubes $M(Q)$ such that for each $Q \in\mathcal{D}_E$  one has  $ M(Q) \subset Q$ and $\eta | M(Q)| \geq |Q|$ with the  constant $\eta$. The  problem is to choose  the disjoint family $\{M(Q)\}_{Q \in \mathcal{D}_E}$. In fact, the argument is an elementary choice of one or two the largest free cubes for each  $Q \in \mathcal{D}_E$. We start with the lemma.
 \begin{lemma}\label{lem}
   Let  $E$ be a porous set with the porosity constant $\eta$. Given  a cube $R\in \mathcal{D}$. Then there is a  disjoint family of free cubes
   \[\mathcal{M}(R) = \{M(Q) \subset Q: Q \in \mathcal{D}_E(R);\;  2^d \eta | M(Q)| \geq |Q|\}.\]
    \end{lemma}
   \begin{proof}
    For $j=0,1,\dots$ let
\[\mathcal{D}_{j,E}(R)=\{Q \in \mathcal{D}_j(R): \:  Q \cap E \neq \varnothing\}.\]
Arguing by induction,  assume that for an integer $ N\geq 0$ there exists a pairwise disjoint family of free cubes
    \[ \mathcal{M}_N(R)=\{M(Q): \;  Q \in \bigcup _{0 \leq n \leq N}\mathcal{D}_{n,E}(R)\},\]
     with the properties
      \begin{enumerate}
         \item[(a)] $ 2^d \eta | M(Q)| \geq |Q|$;
          \item [(b)] Each $Q \in \mathcal{D}_{N,E}(R)$  in addition to    $M(Q)$ may contain no more than one cube $M(Q')$, assigned to  some ancestor $Q' \in \mathcal{D}_{n,E}(R)$, $ 0 \leq n <N$, $Q' \supset Q$.  In this  case there are two distinct children $c_1Q, c_2Q \in  \mathcal{D}_1 (Q)$  such that $c_1Q\supset M(Q)$ and    $c_2Q\supset M(Q')$, respectively.
  \end{enumerate}
 Observe, that the statement is obvious for $N=0$, since we simply take a maximal free cube $M(R)$.

   We check the properties (a)--(b) for $N+1$. For this consider   $\mathcal{D}_1 (Q) \subset\mathcal{D}_{N+1}(R)$ for each $Q \in \mathcal{D}_N(R)$ and choose the assigned free cube for each child in $ \mathcal{D}_{1,E} (Q)$.

   For those children $cQ \in \mathcal{D}_{1,E} (Q)$ such that  $cQ \neq c_1Q$ and  $cQ \neq c_2Q$, we simply choose the largest  free cube $M(cQ) \subset cQ$, $\eta |M(cQ)| \geq  |c Q|$.

   Then, consider the marked children  $c_1Q $ and $c_2Q$, if the latter exists.
    The argument is similar in  both cases.

 Take $c_1Q$ say,  by induction assumption, it already contains free cube
    $c_1Q \supset  M(Q)$.   If    $c_1Q=M(Q)$, that is $c_1Q \notin \mathcal{D}_{1,E} (Q)$, then there is nothing more to be done for $c_1Q$.

  Alternatively,
  if $M(Q) \varsubsetneq c_1Q$,  and recalling that  $M(Q)$ is already assigned to $Q$, we need to find the second  largest free cube $M(c_1Q) \subset c_1Q$  in order to assign it to $c_1Q$.
                     For this  consider  the family $\mathcal{D}_1(c_1Q) \subset \mathcal{D}_{N+2}(R)$. Take a sub-child cube  $ c(c_1Q) \in  \mathcal{D}_{1,E}(c_1Q)$ such that   $c(c_1Q) \cap M(Q) = \varnothing$.
  Choose the largest free cube in  $c(c_1Q)$, but assign it to $c_1Q$,  thus,  denoting it  by $M(c_1Q)$. It holds
     \[\eta |M(c_1Q)| \geq |c(c_1Q)|= 2^{-d}  |c_1Q|\]
    and thus,
     \[2^d \eta |M(c_1Q)| \geq    |c_1Q|\]
     with
 worsening a porosity constant.

     Arguing in the same way for $c_2Q$, we obtain
   all properties (a)--(b) for all cubes from $Q \in \mathcal{D}_{N+1,E}(R)$.
The proof of the lemma is completed.
   \end{proof}
 The next lemma is proved similarly by induction.
\begin{lemma}\label{lem1}
 Let  $E$ be a porous set with a porosity constant $\eta$. Given  a cube $R\in \mathcal{D}$ and a marked free cube
   $R'\subset R$ such that $\eta\,| R'| \geq  | R|$. Then there is a  disjoint family of free cubes
   \[\mathcal{M}(R) = \{M(Q):\;\; Q \in \mathcal{D}_{E}(R) \}\]
   of cubes such that  $2^d \eta \,| M(Q)| \geq |Q|$, and $R' \notin \mathcal{M}(R) $.
\end{lemma}

 To finish the proof of (i)$\Rightarrow$ (iii)
fix any cube $R \in \mathcal{D}_E$, apply
 Lemma \ref{lem} and obtain the claimed  disjoint family
   $\mathcal{M}(R)$ of free cubes.
Then,   extend up the construction to dyadic cubes containing $R$. Take the cube $\pi R \in \mathcal{D}_E$, for which  we need to choose and assign the largest cube $M(\pi R)$. It may be that  the largest cube of $\pi R$  is already assigned to $R$. In any case,  we choose the largest free cube in  a child  $ c\pi R \in \mathcal{D}_1(\pi  R) \setminus R$, denoting  $M(\pi R)$ and   assigning it to  $\pi R$.
It holds $2^d \eta | M(\pi R)| \geq |\pi R|$.

     Applying Lemma \ref{lem1} for this child $c\pi R$ and Lemma \ref{lem}  for the rest   children  in  $\mathcal{D}_1(\pi  R) \setminus R$, we  construct a  disjoint family
   \[\mathcal{M}(\pi R) = \{M(Q):\;\; Q \in \mathcal{D}_{E}(\pi R) \}\]
   of free cubes such that $2^d \eta\: | M(Q)| \geq |Q|$.

   We iterate putting
$R:=\pi  R$.
       The process is infinite, however,    by the property DL3 of lattice,  for each cube $Q \in \mathcal{D}_E$ after finite number of iterations depending on $Q$,  we define a free cube  $M(Q)\subset Q$. By the construction,  the family $M(Q)$ is disjoint,
  and $2^d \eta \:| M(Q)| \geq |Q|$.

   This completes the proof with the sparseness constant $\lambda = 2^d \eta$ provided $\eta$ is a porosity constant,
since the inverse implication
(ii)$\Rightarrow$ (i) is obvious.

\subsubsection{}
(i)$\Leftrightarrow$ (iii)  may be proved alternatively by the Ivrii-Nicolau argument \cite[Lemma 3.1]{IN}. For the sake of completeness and in order to illustrate    the well-sparseness property, we give the proof.

Thus, 
 by the well-sparseness property (ii), there is a collection  of
    pairwise disjoint free cubes  $ M(Q) \subset Q$, $ Q \in \mathcal{D}_E(R)$
  such that    $2^d\eta|M(Q)|\geq  |Q|$, where $\eta$ is a  porosity constant. Consequently,
   \[
\sum_{Q \in \mathcal{D}_E(R)}|Q|^{1-\alpha /d}\leq (2^d\eta)^{\alpha/d-1} \sum_{Q \in \mathcal{D}_E(R)}|M(Q)|^{1-\alpha/d}.
\]
  In order to separate from $E$  for each $M(Q)$ choose a child cube $mQ \in \mathcal{D}_2(M(Q))$. Since $mQ \cap E =\varnothing$, then

 \begin{align}
   \ell (mQ)=1/4\; \ell (M(Q),\label{e_d}\\
    \mathrm{dist} (mQ, E) \geq 1/4 \ell (M(Q)), \label{e_d1}
 \end{align}
  where we calculate in   $l_{\infty}$-metric. Combining (\ref{e_d}) and disjointness
  of the family $m(Q)$, it holds
    \[
   \aligned
\sum_{Q \in \mathcal{D}_E(R)}|Q|^{1-\alpha/d}
&\leq (8^d \eta)^{\alpha/d-1} \sum_{Q \in \mathcal{D}_E(R)}|m(Q)|^{1-\alpha/d}\\
& =(8^d \eta)^{\alpha/d-1} \sum_{Q \in \mathcal{D}_E(R)}\int_{m(Q)}|m(Q)|^{-\alpha/d} dx\\
& = (8^d \eta)^{\alpha/d-1} \sum_{Q \in \mathcal{D}_E(R)}\int_{m(Q)}\ell(m(Q))^{-\alpha } \;dx\\
&\leq(8^d \eta)^{\alpha/d-1}\sum_{Q \in \mathcal{D}_E(R)}\int_{m(Q)}\mathrm{dist} (x,E)^{-\alpha } \;dx\\
& \leq(8^d \eta)^{\alpha/d-1}\int_{R}\mathrm{dist} (x,E)^{-\alpha }\; dx
\;\leq \;C (8^d \eta)^{\alpha/d-1} |R|^{1-\alpha/d}\\
\endaligned
\]
where in the last line Dyn'kin's inequality (\ref{e_i}) with the constant $C=C(E,\alpha)$  is applied.
Thus, (\ref{e_eqvi}) follows.

Conversely,   the claim   follows modulo
 Jensen's inequality, because  for $ \alpha =0 $  the porosity of $E$ implies $|\overline{E}|=0$, and hence \cite[Theorem 7]{V} is applicable.

\subsubsection{}
(i)$\Rightarrow$(iv) follows from (i)$\Rightarrow$(iii) because for porous $E$ one has 
 \begin{equation}\label{ddd}
       |Q|^{1-\alpha /d}\approx \mu_{\alpha,E}(Q) 
     \end{equation}
uniformly with respect to $ Q \in \mathcal{D}_E$.

Conversely, in order  to prove  (iii)$\Rightarrow$(i),  suppose that  $|\overline{E}|=0$ and let
 $ \alpha >0 $ be such that  property   (\ref{e_dyn}) holds.
We clealy have   for each cube $Q \in \mathcal{D}_E$
 \begin{equation}\label{ccc}
  \mu_{\alpha,E}(Q) = \sum _ {Q' \in \mathcal{F}_E(Q)}\mu_{\alpha,E}(Q'),
\end{equation}
where instead of (\ref{ddd}), it holds
\[|Q'|^{1-\alpha /d}\approx \mu_{\alpha,E}(Q')\] 
 uniformly for all $ Q' \in \mathcal{F}_E$. However this observation alone is not sufficient to complete the proof.

 We follow an argument from \cite[Proposition 1]{V}.
Applying   (\ref{ccc}) for each term in the  sum on the left hand-side of (\ref{e_dyn}) and
 changing the summation order, we obtain
\[
\aligned
\sum _ {Q \in \mathcal{D}_E(R)}\mu_{\alpha,E}(Q)
=\sum_ {Q \in \mathcal{D}_E(R)}& \quad\sum_{ Q' \in \mathcal{F}_E(Q)  } \mu_{\alpha,E}( Q') \\
&=\sum_ {Q'\in \mathcal {F}_E(R)} \mu_{\alpha,E}( Q') \quad\sum_{ \substack { Q \in  \mathcal{D}_E(R)\\
  Q\supset Q'}} 1 .
\endaligned
\]
The inner sum above
  reduces to a number  of cubes from $\mathcal{D}_E(R)$ containing $Q'$. This equals  the  number of embedded dyadic cubes  in the sequence
  \[  \pi Q'\subset \pi^2 Q'\subset \dots \pi^n Q'.\]
Since   $\pi^n Q' = R$ is the largest cube in the sequence, we have
\[
 \sum_{ \substack { Q \in  \mathcal{D}_E(R)\\
  Q\supset Q'}} 1  = \log_2 \frac{\ell(R)}{\ell(Q')}.
\]
Thus
\[
\aligned
\sum _ {Q \in \mathcal{D}_E(R)}\mu_{\alpha,E}(Q)
&=\sum_ {Q'\in \mathcal {F}_E(R)} \mu_{\alpha,E}( Q') \log_2 \frac{\ell(R)}{\ell(Q')}.\\
\endaligned
\]
     By assumptions  (\ref{e_dyn}) of the  theorem and then again by identity (\ref{ccc}) applied to the cube $R$, we state
\[
\aligned
\sum_ {Q'\in \mathcal {F}_E(R)} \mu_{\alpha,E}( Q') \log_2 \frac{\ell(R)}{\ell(Q')}
& \leq \; \xi\; \mu_{\alpha,E}(R)\\
& = \;\xi\; \sum_ {Q'\in \mathcal {F}_E(R)} \mu_{\alpha,E}( Q')
\endaligned
\]
Choose   the largest free cube  $M(R) \in \mathcal{F}_E(R)$  such that $\ell(M(R)) \geq \ell(Q')$ for all cubes $ Q'\in \mathcal {F}_E(R)$, then
\[\log_2 \frac{\ell(R)}{\ell(M(R))}\leq \xi, \]
 and porosity of $E$ follows with the constant $2^\xi$.
The proof of Theorem   \ref{t1} is completed.

\subsection{Proof of Corollary \ref{eqvi}}
  By the definition of the Aikawa--Assouad codimension it is required to consider only  $\mathrm{codim}_A E >0$.
 In this case the proof of Theorem \ref{t1} (part (i)$\Rightarrow$ (iii) states   that for each $0\leq \alpha <\mathrm{codim}_A E$ the
property   (\ref{e_dyn}) holds.

Conversely,   put
$ \alpha >0 $ be such that (\ref{e_dyn})  holds.   Theorem \ref{t1}  states the porosity of  $E$. Thus by Theorem  \ref{t1}  it does not matter what a kind of estimates
 (\ref{e_dyn}) or (\ref{e_eqvi}) is to use.  Therefore, say in the case of (\ref{e_eqvi}), it holds
 \[
   C |R|^{1-\alpha/d}
\geq \sum_{Q \in \mathcal{D}_E(R)}|Q|^{1-\alpha/d}.
\]
 Now let  $\pi Q' \in \mathcal{D}_E(R)$ be assigned to $Q' \in \mathcal{F}_E(R)$.
 In fact, the cube $\pi Q'$  may be assigned to  other its children   $c (\pi Q') \in \mathcal{F}_E(R)$ such that $\pi Q' =\pi (c (\pi Q'))$.
 The number of all such children is estimated from above by  $2^d-1$. Therefore, the sum
 \[\sum_{Q' \in \mathcal{F}_E(R)}|\pi Q'|^{1-\alpha/d}\]
  may contain each summand  $|\pi Q'|^{1-\alpha/d}$  with multiplicity no more than  $2^d$. Thus, we estimate
 \[
   \aligned
  \sum_{Q \in \mathcal{D}_E(R)}|Q|^{1-\alpha/d}
&\geq 2^{-d} \sum_{Q' \in \mathcal{F}_E(R)}|\pi Q'|^{1-\alpha/d}.\\
\endaligned
\]
Observing  $|\pi Q'|= 2^d |Q'|$, we combine all the estimates
\[
   \aligned
C |R|^{1-\alpha/d}
&\geq  2^{-d}\sum_{Q' \in \mathcal{F}_E(R)} \;2^{d(1-\alpha/d)} \;| Q'|^{1-\alpha/d},
\endaligned
\]
which yields     Dyn'kin's inequality (\ref{aaa}) with a constant $ 2^{\alpha} C$. Thus, the claim  $\alpha < \mathrm{codim}_A E$ follows.

\subsection{ Carleson packing condition  implies  porosity. Proof   of Theorem \ref{t2}}

In order to  find a  set $E$ such that   $ \mathcal{S} \subset  \mathcal{D}_E$, one claims,  particularly,
 $Q \cap E \neq \varnothing $ for each $Q \in \mathcal{S}$.
 For each such cube $Q$ consider an infinite chain
\[\mathcal{P}(Q)=\{Q \supset cQ\supset
c^2Q \dots \}\subset \mathcal{D}(Q), \]
 where $c^k Q \in \mathcal{D}_k(Q) $  have the same left-down corner with $Q$. Namely, if
 \[Q= [a_{Q,1}, b_{Q,1}) \times \dots \times [a_{Q,d}, b_{Q,d})\]
  then   $c^k Q \ni a_Q=(a_{Q,j})_{j=1,\dots, d}$ for each $k=1, 2, \dots$. Though, some cubes $c^kQ$ can already belong to $\mathcal{S}$, this is not an issue for us.
 Let $E$ be a set of left-down corners  $a_Q$ of cubes $Q \in \mathcal{S}$. Clearly, $ \mathcal{S} \subset  \mathcal{D}_E$.

We prove that the  collection $\mathcal{D}_E$ defined by the constructed set $E$ satisfies to the Carleson packing condition.
Indeed, given a cube $R \in \mathcal{D}$ consider the family of cubes $Q' \in \mathcal{D}_E$, $Q' \subset R$. It holds

\[
\aligned
\sum_{\substack {Q' \in \mathcal{D}_E\\ Q' \subset R}}|Q'|
&= \sum_{\substack {Q' \in \mathcal{S}\cap \mathcal{D}_E\\ Q' \subset R}} |Q'| + \sum_{\substack {Q' \in  \mathcal{D}_E \setminus \mathcal{S}\\ Q' \subset R}} |Q'|\\
&=S_1\quad +\quad S_2.
\endaligned
\]
Since cubes from the first sum on the right hand side belong to $ \mathcal{S}$,  hence by assumptions of the theorem, we estimate this sum as
\[
S_1= \sum_{\substack {Q' \in \mathcal{S}\cap \mathcal{D}_E\\ Q' \subset R}} |Q'|\leq \xi |R|.
\]
with the Carleson constant $\xi$.

In order to estimate $S_2$, note that if $Q' \in \mathcal{D}_E \setminus \mathcal{S}$, then there exists
$Q \in \mathcal{S}$ such that its left-down corner  $a_Q $ belongs to $Q'$. Therefore,  $Q\cap Q' \neq \varnothing$,
and by the lattice property (\ref{dyp}), one has $Q\cap Q'= \{Q,\;Q'\}$.

We have $Q'\subset Q$. Indeed, in the case  $Q \subset  Q'$, one has  $Q \in \mathcal{D}(Q')$. Therefore,  by the assumptions of the theorem,  $Q \in \mathcal{S}$ implies  $Q' \in \mathcal{S}$, which contradicts to embedding $Q' \in \mathcal{D}_E \setminus \mathcal{S}$.

Thus, $Q'\subset Q$ implies   $Q'\in \mathcal{D}(Q)$. Since $a_Q \in Q'$, hence $a_Q$ is a left-down corner of $Q'$, as well, and $Q'$ belongs to a
 chain $\mathcal{P}(Q)$.
Therefore, we  write the second sum as a double sum
\[
\aligned
S_2
&=\sum_{\substack {Q' \in  \mathcal{D}_E \setminus \mathcal{S}\\ Q' \subset R}} |Q'|
 =\sum_{Q \in \mathcal{S}} \quad\sum_{\substack {Q'\in  \mathcal{P}(Q) \\Q'\subset R}}  |Q'|.
\endaligned
\]
Split the  outer sum according to wether $Q\subset R$, or $R\subset Q$, and obtain
\[
\aligned
S_2
=\sum_{\substack {Q \in \mathcal{S}\\ Q \subset R}} \quad\sum_{Q' \in  \mathcal{P}(Q)}  |Q'|+
\sum_{\substack {Q \in \mathcal{S}\\ R\subset Q}} \quad\sum_{\substack {Q' \in  \mathcal{P}(Q)\\ Q'\subset R}}& |Q'|\\
&=S_3\; + \; S_4.
\endaligned
\]
To estimate the  sum $S_4$,  observe that since $Q' \subset R \subset Q$, the left-down corner $a_Q$ of $Q$ coincides with the  left-down corner  of $R$. Therefore,  the double sum boils down to the unique inner sum
\[
S_4=\sum_{\substack {Q \in \mathcal{S}\\ R\subset Q}} \quad\sum_{\substack {Q' \in  \mathcal{P}(Q)\\ Q'\subset R}} |Q'|=
\sum_{Q' \in  \mathcal{P}(R)} |Q'|.
\]
By the  geometric progression formula,
\[S_4 \leq \sum_{Q' \in  \mathcal{P}(R)} |Q'|\leq \frac{2^d}{2^d-1}|R|.\]
Also by the  geometric progression formula, estimating the inner sum  in $S_3$,  and then by the Carleson packing condition with the constant $\xi$, it holds

\[
\aligned
S_3
 =\sum_{\substack {Q \in \mathcal{S}\\ Q \subset R}} \quad &\sum_{Q' \in  \mathcal{P}(Q)}  |Q'|\\
&\leq \sum_{\substack {Q \in \mathcal{S}\\ Q \subset R}} \quad \frac{2^d}{2^d-1}|Q|
\leq \frac{2^d \xi }{2^d-1}|R|.
\endaligned
\]
   Thus, the family $\mathcal{D}_E$ is Carleson with a constant
   \[C(\xi)=\xi + \frac{2^d}{2^d-1}+ \frac{2^d}{2^d-1} \xi. \]
     Observing   (i)$\Leftrightarrow $(ii) of Theorem \ref{t1}, we obtain  that $E$ is porous, which completes the proof.

\subsection{ Carleson property of  $ \mathcal{D}_{\gamma, E}$. Proof   of Theorem \ref{t3}}
 \subsubsection{}
  (ii)$\Rightarrow$(i) is easy. Indeed, since  $ \mathcal{D}_{ E} \subset \mathcal{D}_{\gamma, E}$ the family $\mathcal{D}_{ E}$ is Carleson, and the porosity of $E$ follows from \cite{BNT}, \cite{V}.

 \subsubsection{} Conversely, for (i)$\Rightarrow$(ii)
choose   the minimal integer $n$ such that
$n > \gamma$.
Fix a cube $R \in \mathcal{D}$. Let $\mathcal{D}_{\gamma, E}(R) = \{Q \in \mathcal{D}_{\gamma, E},\;
 Q \subset R \}$.
   For each cube $Q \in \mathcal{D}_{\gamma, E}(R)$
 define the dilated cube $\widetilde{Q}=  (2n+1) Q$ with the same center. Clearly, $\widetilde{Q} \subset  \widetilde{R}$, where  $\widetilde{R} =(2n+1) R $ is the dilated cube to $R$. By the definition of the family $\mathcal{D}_{\gamma, E}$, it holds that $\widetilde{Q} \cap E \neq \varnothing$, although it may be that  $\widetilde{Q} \notin  \mathcal{D}$.  The cube $\widetilde{Q}$ consists of $(2n+1)^d$ cubes of size $\ell(Q)$.
Among these cubes there is  at least  a cube $Q'$ such that $Q' \in \mathcal{D}_E$.

 On the other hand for each cube $Q'\in \mathcal{D}_E$ such that $Q'  \subset  \widetilde{R} $, there is no more than $(2n+1)^d$ cubes $Q \in \mathcal{D}_{\gamma, E}(R)$ such that $\ell(Q)=\ell(Q')$ and  $Q' \subset \widetilde{Q}=(2n+1) Q $. Therefore, we estimate the number of cubes
 $Q \in \mathcal{D}_{\gamma, E}(R)$ of the given size $\ell(Q)$ by $(2n+1)^d$ times the number of cubes
 $Q' \subset \widetilde{R} $,  $Q' \in \mathcal{D}_E$ with $\ell(Q')=\ell(Q)$.  Thus, it holds
\[
 \aligned
\sum_{Q \in
\mathcal{D}_{\gamma, E}(R)} \quad |Q|
  &\leq (2n+1)^d \sum_{ \substack {Q' \in\mathcal{D}_{ E}\\
  Q' \subset\widetilde{R}}} \quad|Q'|.
           \endaligned
  \]
  Although it may be  $\widetilde{R} \notin \mathcal{D} $, we choose the maximal cube $M(R) \in \mathcal{D}$,  $M(R) \subset \widetilde{R}$ and $ \ell(\widetilde{R})\leq 2 \ell(M(R)) < 2\ell(\widetilde{R}) $.   We easily cover $\widetilde{R}$  by  $N \leq 3^d$ dyadic cubes
  $R_i \in \mathcal{D} $  of the size $ \ell(R_i) = \ell(M(R)) $ as
 $\widetilde{R} \subset \bigcup_{i=1}^{N} R_i $.

   Since  \cite[Theorem 7]{V} the family $\mathcal{D}_{ E}$  is Carleson with   the  Carleson constant $C=C(E)$, therefore, the latter sum is estimated
  \[
 \aligned
  \quad  \sum_{ \substack {Q' \in\mathcal{D}_{ E}\\
  Q' \subset\widetilde{R}} } \;|Q'|
   \leq  \sum_{i=1} ^{N} &\sum_ {Q' \in\mathcal{D}_E( R_i)}  \;|Q'|\\
    &\leq C\;  \sum_{i=1} ^{N}|R_i| \;\leq C\; 3^d \; |\widetilde{R}|.
        \endaligned
     \]
 Observing  $|\widetilde{R}| =  (2n+1)^d |R|$, we  obtain
   the Carleson property of
the family $\mathcal{D}_{\gamma, E}$ with the constant $C(\gamma +1)^{d}6^d$, and (i)$\Rightarrow$(ii) follows.

 \subsubsection{} Since (iii)$\Rightarrow$(ii) is obvious, it remains to  prove
   (ii)$\Rightarrow$(iii).
Check that the family $\mathcal{D}_{\gamma, E}$ satisfies the assumptions of Theorem \ref{t2}.
Indeed,  $\mathcal{D}_{\gamma, E}$ is Carleson by assumption and   for each cube $Q \in \mathcal{D}_{\gamma, E}$
one has obviously $\pi Q \in \mathcal{D}_{\gamma, E}$.

Consequently, Theorem \ref{t2} implies that
there is a porous set $\widetilde{E}\in  \rd$ such that
$\mathcal{D}_{\gamma, E}  \subset \mathcal{D}_{\widetilde{E}}$.
 Observing that $\widetilde{E}$ consists of all left-down corners $a_Q$ of cubes $Q \in
\mathcal{D}_{\gamma, E}$, and since each point in $E$ just coincides with the left-down corner of the certain   cube $Q \in \mathcal{D}_E$, we conclude that $E \subset \widetilde{E}$.

By Theorem \ref{t1} $\mathcal{D}_{\widetilde{E}}$ is well-sparse with respect to $\widetilde{E}$. This means that for each  $Q \in \mathcal{D}_{\gamma, E} \subset \mathcal{D}_{\widetilde{E}}$ one can choose a  disjoint family of cubes $M(Q) \subset Q$ such that  $M(Q) \cap \widetilde{E} = \varnothing$
and $\lambda| M(Q)| \geq |Q|$ with $\lambda>1$ depending on $\widetilde{E}$ and $\gamma$.
Since $E \subset \widetilde{E} $ implies $M(Q) \cap E =\varnothing $, the claim
  follows.

Thus,  the proof of the theorem is completed.

\subsection{Proof of Corollary \ref{cor3}}
\subsubsection{}

  The  standard approach  to prove the equivalence of (iii) in Corollary \ref{cor3} and (ii) in Theorem \ref{t3} are  the dyadic   arguments  \cite[p. 59, Lemma 5]{MC} together  with   the maximal function techniques  \cite[Theorem 2.10]{IV}. However, with  the well-sparseness property in hand this one is obvious.

 Also if (i) or (ii) holds, then  the embedding $\mathcal{D}_{ E}  \subset \mathcal{D}_{\gamma, E} $ implies the  porosity of $E$ by Theorem \ref{t1}. Thus it remains to prove (i) and (ii).

\subsubsection{} Proof of (i).
Let $E$ be porous.
      The proof of Theorem \ref{t3} (ii)$\Rightarrow$(iii) gives a porous set $\widetilde{E} \supset E$ such that  $\mathcal{D}_{\gamma, E}  \subset \mathcal{D}_{\widetilde{E}}$.
   Then by Theorem \ref{t1} there exist constant $\alpha >0$ and $\zeta >0$ depending on $\widetilde{E}$ such that for each $R \in \mathcal{D}_{\widetilde{E}}$
    \[
 \aligned
   \sum_{ Q \in\mathcal{D}_{\gamma, E}(R) } \;|Q|^{1-\alpha/d}
   \leq  \sum_{ Q \in\mathcal{D}_{\widetilde{E}}(R) }& \;|Q|^{1-\alpha/d}\\
    &\leq \zeta |R|^{1-\alpha/d}.
        \endaligned
     \]
 This completes the proof, since if $R \notin \mathcal{D}_{\widetilde{E}}$, then the set  $\mathcal{D}_{\widetilde{E}}(R)$ is empty, and (i) holds trivially.
  \subsubsection{} Proof of (ii). Repeating the argument  (ii)$\Rightarrow$(iii) of Theorem \ref{t3} we obtain  a porous set  $\widetilde{E} \supset E$ such that
     $\mu_{\alpha,E}(Q) \leq \mu_{\alpha,\widetilde{E}}(Q)$ for each $Q \in \mathcal{D}$. Then applying Theorem \ref{t1} with porous $\widetilde{E}$ we get that  there exist two constants $\alpha >0$ and $\widetilde{\xi} >0$ such that
    \[
\aligned
\sum _ {Q \in \mathcal{D}_{\gamma, E}(R)}\mu_{\alpha,E}(Q)
&\leq \sum _ {Q \in \mathcal{D}_{\widetilde{E}}(R)}\mu_{\alpha,\widetilde{E}}(Q)\\
& \leq \widetilde{\xi} \mu_{\alpha,\widetilde{E}}(R)
\leq \;C \;\widetilde{\xi} \; |R|^{1-\alpha /d},\\
 \endaligned
\]
where in the last line we apply Dyn'kin's inequality (\ref{e_i}) with the constant $C>0$ independent of $R$.

  Now if $R \in \mathcal{D}_{\gamma, E}$, then by the elementary estimate for    $x \in R$
  \[\mathrm{dist}(x, E)\leq (\gamma +2)\ell(R),\]
 we have
   \[|R|^{1-\alpha /d}\leq \int_{R} \frac{(\gamma +2)^{\alpha}}{\mathrm{dist}^{\alpha}(x,E)}dx =
    (\gamma +2)^{\alpha}\mu_{\alpha,E}(R).\]
 That is required for (ii) with the constant $\xi=  (\gamma +2)^{\alpha} C \;\widetilde{\xi}$,
 since for $R \notin \mathcal{D}_{\gamma, E}$ the set  $\mathcal{D}_{\gamma, E}(R)$ is empty,
and (ii) holds trivially.

%\subsection{Ethical statement} The author declares no conflict of interest. This study did not involve human participants. No new data was created or analyzed in this study.

 %\appendix{\textbf{Thanks.}}

\bibliographystyle{amsplain}

\end{document}